\documentclass[12pt]{article}
\usepackage[margin=0.8in]{geometry}
\usepackage{amsmath,amssymb,amsthm,amsfonts,color,hyperref,graphicx,subfigure,authblk}
\usepackage{mathrsfs}
\usepackage{indentfirst}
\usepackage{comment}
\usepackage{cleveref}

\makeatletter
\def\blfootnote{\gdef\@thefnmark{}\@footnotetext}
\makeatother

\hypersetup{colorlinks,
linkcolor=blue} 

\theoremstyle{definition}
\newtheorem{thm}{Theorem}[section]
\newtheorem{lem}[thm]{Lemma}
\newtheorem{rem}[thm]{Remark}
\newtheorem{defi}[thm]{Definition}
\newtheorem{prop}[thm]{Proposition}
\newtheorem{cor}[thm]{Corollary}

\numberwithin{equation}{section}

\DeclareMathOperator{\Graph}{Graph}

\providecommand{\keywords}[1]
{
  \textbf{\text{Keywords: }} #1
}
\newcommand{\enabstractname}{Abstract}
\newenvironment{enabstract}{
    \par\small
    \noindent\mbox{}\hfill{\bfseries \enabstractname}\hfill\mbox{}\par
    \vskip 2.5ex}{\par\vskip 2.5ex} 

\title{Stability of geodesic-ray data, horofunctions, and rectifiability of fixed-shape slices in the Newtonian \(N\)-body problem}
\author[1]{Putian Yang \footnote{Email: PutianYang@outlook.com}}
\author[2]{Shiqing Zhang\thanks{Corresponding author: zhangshiqing@scu.edu.cn}}

\affil[1]{Chongqing University of Science and Technology}
\affil[2]{Sichuan University}

\date{}

\begin{document}
\maketitle

\blfootnote{\textup{2020} \textit{Mathematics Subject Classification}:
Primary 70F10, 70H20, 49L25; Secondary 37J50, 28A75.}

\begin{enabstract}
For the Newtonian \(N\)-body problem at nonnegative energy, we study solution sets selected by the Jacobi--Maupertuis variational principle and by 
the associated stationary Hamilton--Jacobi equation. We prove a compactness/stability theorem for classical initial data generating geodesic rays: 
limits in the ambient phase space remain collision-free, generate geodesic rays, and carry locally relatively compact normalized Busemann functions.
The limiting horofunction of normalized Busemann functions yields a viscosity solution of the limiting stationary equation. 
For a fixed collision-free hyperbolic limit shape \(a\), we also prove closedness of the corresponding slice of geodesic-ray data. Finally, after passing
 to the reduced configuration space \(X\), we show that this fixed-shape slice is countably \(d(N-1)\)-rectifiable in phase space and has Hausdorff dimension exactly \(d(N-1)\). Thus the paper combines phase-space compactness of calibrated minimizing motions with a geometric-measure description of a fixed-shape Hamilton--Jacobi calibrated slice.
\end{enabstract}


\section{Introduction}\label{Introduction}

Let $E=\mathbb{R}^d$ with $d\ge 2$, and consider the Newtonian $N$-body problem in the
configuration space $E^N$, endowed with the mass inner product
\[
\langle x,y\rangle=\sum_{i=1}^{N}m_i(x_i,y_i),
\qquad x=(x_1,\dots,x_N),\ y=(y_1,\dots,y_N)\in E^N,
\]
where $(\cdot,\cdot)$ is the standard inner product in $E$.
Taking suitable units such that the gravitation constant $G=1$, the Newtonian potential reads
\[
-U(x)= -\sum_{1\le i<j\le N}\frac{m_i m_j}{|x_i-x_j|},
\]
defined on the collision-free set
\[
\Omega=\{x\in E^N:\ x_i\neq x_j,\ 1\le i<j\le N\},
\qquad
\Delta=E^N\setminus\Omega.
\]
The classical phase space is \(T\Omega\simeq \Omega\times E^N\). Throughout the first part of the paper, the larger Euclidean space \(TE^N\simeq E^N\times E^N\) is used only as an ambient space when discussing compactness and closure properties. If \(x_0\in\Delta\), the pair \((x_0,v_0)\) is not a classical Cauchy datum for Newton's equations. In this case, saying that a motion starts from \(x_0\) means that it is a continuous half-entire curve \(x:[0,+\infty)\to E^N\) such that \(x(0)=x_0\), \(x(t)\in\Omega\) for all \(t>0\), and \(x|_{(0,+\infty)}\) is a classical solution of Newton's equations.

For each nonnegative energy level \(h\ge 0\), the Jacobi--Maupertuis distance \(\phi_h\) is obtained by minimizing the free-time action \(A_h\). Geodesic rays for \(\phi_h\) correspond to free-time minimizing motions of energy \(h\). This variational-geometric viewpoint is closely related to the stationary Hamilton--Jacobi equation
\begin{equation}\label{HJ equation intro}
H(x,D u)=h,
\qquad
H(x,p)=\frac12\|p\|_*^2-U(x),
\end{equation}
and to weak KAM theory for singular \(N\)-body Hamiltonians \cite{maderna_2012,10.4007/annals.2020.192.2.5,Burgos2022}. In particular, dominated functions are viscosity subsolutions, and curves calibrating such functions are action-minimizing characteristics.

The boundary-at-infinity objects used below are normalized horofunctions. For \(h\ge 0\), we denote by \(\mathscr B_h\) the class of all locally uniform limits of normalized functions of the form
\[
x\longmapsto \phi_{h_n}(0,p_n)-\phi_{h_n}(x,p_n),
\qquad
h_n\ge 0,
\quad h_n\to h,
\quad \|p_n\|\to +\infty.
\]
When the points \(p_n\) are taken along a geodesic ray \(\gamma\) at fixed energy, the resulting limit is the normalized Busemann function
\[
\widehat b_\gamma(x)
=
\lim_{t\to+\infty}\bigl(\phi_h(0,\gamma(t))-\phi_h(x,\gamma(t))\bigr).
\]
Thus exact normalized Busemann functions form a distinguished subclass of the horofunction class \(\mathscr B_h\). In this paper, the phrase ``Busemann data'' refers to this boundary-at-infinity structure, while the notation \(\mathscr B_h\) is reserved for the larger horofunction closure class.

Recent progress has clarified several complementary aspects of this noncompact weak KAM picture. Maderna developed weak KAM theory for \(N\)-body problems and proved the existence of global viscosity solutions through the Lax--Oleinik semigroup \cite{maderna_2012}. Maderna and Venturelli later constructed hyperbolic motions with arbitrary initial configuration, prescribed collision-free limit shape, and prescribed positive energy by means of global viscosity solutions of the Hamilton--Jacobi equation \cite{10.4007/annals.2020.192.2.5}. Polimeni and Terracini established, by a renormalized variational principle, the existence of minimal expansive half-entire solutions in the hyperbolic, parabolic, and hyperbolic-parabolic regimes; the corresponding value functions yield viscosity solutions after a linear correction \cite{Polimeni-Terracini-2024}. In the hyperbolic case, Maderna and Venturelli proved the uniqueness of the Busemann function associated with a prescribed collision-free limit shape, localized a cone of differentiability, and showed that every hyperbolic motion is eventually minimizing \cite{Maderna-Venturelli-2026}. From the PDE and geometric-measure side, Berti, Polimeni, and Terracini studied renormalized value functions as viscosity solutions and proved rectifiability results for their singular and conjugate sets \cite{Terracini-2025}.

A classical stability result underlying the hyperbolic theory is Chazy's continuity of the limit-shape map. In the formulation recalled by Maderna--Venturelli \cite[Lemma 4.1]{10.4007/annals.2020.192.2.5}, 
the set of hyperbolic initial data is open and the map sending a hyperbolic initial datum to its asymptotic velocity is continuous; 
the 2026 work \cite{Maderna-Venturelli-2026} of Maderna--Venturelli further proves \(C^1\)-regularity of this map on the hyperbolic region. Our compactness question is different but complementary. 
Chazy's lemma states the persistence and variation of the asymptotic velocity among nearby hyperbolic initial data. The present paper gives 
the persistence of the calibrated/geodesic-ray property itself, the exclusion of collision in ambient limits, and the stability of the associated normalized horofunction data. In particular, our first theorem is not merely a continuity statement for the limit shape; it is a compactness theorem for the solution set selected by the Jacobi--Maupertuis variational principle and by Hamilton--Jacobi boundary data.

The present paper mainly studies two related solution sets selected by these variational and Hamilton--Jacobi structures. The first one is the set of geodesic-ray initial data
\[
GR:=\{(x,v)\in T\Omega:\ \text{the maximal classical solution with initial datum }(x,v)
\text{ is a geodesic ray}\}.
\]

Our first main result states that \(GR\) is stable and closed under ambient convergence in \(TE^N\). 
It states that limits of geodesic-ray data remain collision-free and again generate geodesic rays, and that 
the corresponding horofunction data converge, up to subsequences, to a function that still calibrates the limiting ray.

\begin{thm}[Compactness of geodesic-ray data and normalized horofunctions]\label{thm:compactness_GR_Busemann}
Let \((x_n,v_n)\subset GR\) and assume that
\[
(x_n,v_n)\to (x_0,v_0)\qquad\text{in }TE^N.
\]
For each \(n\), let \(\gamma_n:[0,+\infty)\to\Omega\) be the geodesic ray with
\[
\gamma_n(0)=x_n,
\qquad
\dot\gamma_n(0)=v_n,
\]
and let
\[
h_n=\frac12\|v_n\|^2-U(x_n)
\]
be its total energy. Then the following assertions hold.
\begin{enumerate}
\item \(x_0\in\Omega\), and
\[
h_n\to h:=\frac12\|v_0\|^2-U(x_0)\ge 0.
\]

\item The normalized functions
\[
u_n(x):=\phi_{h_n}(0,p_n)-\phi_{h_n}(x,p_n)
\]
admit a locally uniformly convergent subsequence, still denoted by \(u_n\), such that
\[
u_n\to u\qquad\text{locally uniformly on }E^N,
\]
where \(u\in\mathscr B_h\).
\item If \(\gamma:[0,t_*)\to\Omega\) is the maximal classical solution with initial datum
\[
\gamma(0)=x_0,
\qquad
\dot\gamma(0)=v_0,
\]
then \(t_*=+\infty\), the curve \(\gamma\) \(h\)-calibrates \(u\), and therefore \(\gamma\) is a free-time minimizer of \(A_h\). In particular, \(\gamma\) is a geodesic ray.
\end{enumerate}
\end{thm}

As a first refinement, we proved Proposition \ref{prop:precompact_busemann_data}, which concerns normalized Busemann functions
 associated with the rays \(\gamma_n\).
 They are relatively compact in \(C_{\mathrm{loc}}(E^N)\), and their cluster points belong to the horofunction class \(\mathscr B_h\). 
 Moreover, the associated viscosity solutions of $\gamma_n$ are \(-\widehat b_{\gamma_n}\), if \(\widehat b_{\gamma_n}\to u\) and \(h_n\to h\),
  the limit \(-u\) is a viscosity solution of \(H(x,Du)=h\) on \(\Omega\). This makes the compactness theorem a stability statement not 
  only for calibrated motions but also for the corresponding Hamilton--Jacobi boundary data.
For a fixed collision-free hyperbolic limit shape \(a\), the fixed-shape Busemann uniqueness theorem of Maderna--Venturelli \cite{Maderna-Venturelli-2026} allows 
us to prove closedness of the corresponding slice \(GR(a)\subset T\Omega\). 

As our second main result, the fixed-shape $GR(a)$ has a further geometric-measure formulation
 after removing translations. Define the reduced configuration space
\[
X:=\left\{x=(r_1,\dots,r_N)\in E^N:\ \sum_{i=1}^N m_i r_i=0\right\},
\qquad
\Omega_X:=\Omega\cap X,
\]
and write \(TX\simeq X\times X\). 
For any $a\in\Omega_X$, define the reduced fixed-shape slice

\[
GR_X(a):=\{(x,v)\in T\Omega_X:\ \text{the hyperbolic solution starting from }(x,v)
\text{ with limit shape }a\}.
\]
 Thus \(GR_X(a)\) consists of initial data whose corresponding hyperbolic geodesic rays have limit shape \(a\). The theorem of the cone and the regularity of the limit-shape map imply that, on a suitable cutted cone, \(GR_X(a)\) is the graph of a \(C^1\) velocity field. Globalizing this graph structure along the Hamiltonian flow gives the following geometric-measure description.

\begin{thm}[Geometric measure structure of the fixed-shape slice]\label{thm:GRa_geom_measure}
Let \(a\in\Omega_X\). Then \(GR_X(a)\) is countably \(d(N-1)\)-rectifiable in \(TX\). Moreover, the Hausdorff measure
\[
\dim_H GR_X(a)= d(N-1).
\]
\end{thm}

This theorem describes the regular phase-space side of the fixed-shape Hamilton--Jacobi problem. It is complementary to the singular-set theory in \cite{Terracini-2025}: Berti--Polimeni--Terracini study the codimension-one irregular set of renormalized Hamilton--Jacobi value functions in configuration space, whereas \Cref{thm:GRa_geom_measure} describes the rectifiable slice of initial data selected by a fixed-shape Busemann solution in phase space.

The paper is organized as follows. \Cref{Variational background} collects the variational background: action potentials, free-time minimizers, geodesic rays, calibrating curves, and horofunctions. \Cref{sec:compactness} proves \Cref{thm:compactness_GR_Busemann}, the relative compactness of exact normalized Busemann functions, the Hamilton--Jacobi interpretation of their cluster points, and the closedness of fixed-shape hyperbolic slices. \Cref{sec:geometric} proves \Cref{thm:GRa_geom_measure}, the geometric-measure structure of \(GR_X(a)\).

\section{Variational background}\label{Variational background}

\subsection{Action potentials, free-time minimizers and geodesic rays}

The dynamics of the $N$ bodies are ruled by the Newton's universal gravitation law, that is
\begin{equation}\label{Newton motion equation}
    \ddot{x}=\nabla U(x),\text{ for } 1\leq i\leq N,
  \end{equation}
 the gradient is taken w.r.t the mass inner product.

The Lagrangian $L$ on $E^N \times E^N$ is defined as
  \begin{equation}
    L(x,v)=\frac{1}{2}\|v\|^{2}+U(x).
  \end{equation}
  According to  the famous Hamilton's principle, a solution of the Newton's equation joining two arbitrary given $x,y\in E^N$ must be a critical point of the Lagrangian action, and more precisely
  a minimizer of Lagrangian action among space of absolutely continuous curves joining $x$ and $y$ with  fixed time interval. 
  The existence of such minimizers  was already confirmed by Tonelli's theorem, see F. Clarke  \cite[p321, Theorem 16.2]{clarke2013functional}.
  We next give this approach explicitly.
  
  Define the set of absolutely curves
  
  \begin{equation}
     AC  (x,y,T):=\left\{\gamma:[0,T]\rightarrow E^N \text{ is absolutely continuous, } \gamma(0)=x,\gamma(T)=y\right\}
   \end{equation}
   and
   \begin{equation}
     AC  (x,y)=\cup_{T>0}   AC  (x,y,T),
   \end{equation}
   since without loss of generality, we can always assume each curve starting at $0$ ending at some $T>0$ for $T$.

The Hamiltonian $H$ of the Lagrangian is the first integral of the solutions to the Newton's equation, and 
$h$ is the total energy constant.
The Hamilton-Jacobi equation is 
\begin{equation}\label{HJ equation}
    H(x,d_x u)= \frac{1}{2}\|d_x u\|_*^2 -U(x) =h,
\end{equation}
where the norm $\|\cdot\|_*$ is the norm of $(E^*)^N$, for $p\in (E^*)^N$,
\begin{equation}
\|p\|_*^2 = \sum_{i=1}^N m_i^{-1}|p_i|^2.
\end{equation}
The Lagrangian action of $\gamma\in  AC  (x,y)$ with energy constant $h\geq 0$ is

\begin{equation}
    A_h (\gamma)=\int_0^TL(\gamma,\dot{\gamma})+hdt= \int_0^T \frac{1}{2}\|\dot{\gamma}\|^{2}+U(\gamma)+hdt
  \end{equation}
when $h=0$, we simply write $A$. It is well defined since $\Omega$ is connected, and for all $x,y\in\Omega$ there is a smooth $\gamma\in  AC  (x,y,T)$ and $U$ is bounded 
in the compact set $\gamma\vert_{[0,T]}$ hence $A_h (\gamma)<\infty$ and $0\leq A_h\not\equiv+\infty$.
  
  However, once there was a big obstacle for the calculus of variations that under Newton potential there will be trajectory with isolated collisions and its Lagrangian action is finite, 
  this discovery is due to Poincar{\'e} \cite{zbMATH02677123}, leading to a question that a minimizer may not be a true motion.
  But mathematicians did not worry about it long since Marchal   \cite{Marchal2002} have made a great breakthrough,  making sure 
  in case of more than 2 dimensional, the minimization process 
  confirms the absence of collision in the classical Newtonian $N$-body problem, see more in Chenciner \cite{zbMATH01789893}.
  In 1-D case, Yu and Zhang \cite{Yu2018} has obtained important results under some conditions.
We shall use the  classical result of Marchal: if $ \gamma$ minimizes $A$ over $AC(x,y,T)$,
then $\gamma$ has no interior collision; see \cite{Marchal2002,zbMATH01789893}.
  Therefore a minimizer of the action is a solution of the Newtonian equation \ref{Newton motion equation} in $(0,T)$ and a true motion of Newtonian systems.

Let $x,y\in E^N$ be arbitrary, we define the fixed time action potential for time $T$:
\begin{equation*}
    \phi(x,y,T)=\inf \left\{ A(\gamma)\big\vert\gamma\in AC (x,y,T) \right\},
\end{equation*}
and the free-time action potential $ \phi_h (\cdot,\cdot)$:
\begin{equation*}
    \phi_h (x,y)=\inf\{A_h (\gamma)\mid\gamma\in  AC  (x,y)\}.
\end{equation*}
A curve $\gamma_0\in  AC  (x,y,T)$ is called a fixed-time minimizer of $A$ if it minimizes $A$ in $ AC  (x,y,T)$.
Similarly, a curve $\gamma_0\in  AC  (x,y)$ is called a free-time minimizer of $A_h$ if it minimizes $A_h$ in $ AC  (x,y)$, 
the existence of free-time minimizers is proved, see  \cite[Theorem 3.1]{daluz-maderna-2014} and  \cite[Lemma 4.2]{10.4007/annals.2020.192.2.5}.
It is obvious that a free-time minimizer is a fixed time minimizer hence experiences no collision in the middle of the interval.

\begin{prop}\label{free-time minimizer of phi_h exist}
For $h>0$ and $x\neq y$, the infimum in the definition of $\phi_h(x,y$) is attained by some
curve $\gamma\in AC(x,y)$; see \cite[Lemma 4.2]{10.4007/annals.2020.192.2.5}.
For $h=0$ and $x\neq y$, the analogous statement for the critical action potential $\phi(x,y$)
is classical; see da Luz--Maderna \cite[Theorem 3.1]{daluz-maderna-2014}.
\end{prop}
A curve $\gamma$ defined on $[0,+\infty)$ is called a free-time minimizer if it is a free-time minimizer 
of the Lagrangian action when restricted in any compact subinterval $[a,b]\subset [0,+\infty)$,
that is, $\gamma_{[a,b]}$ minimizes Lagrangian over $ AC (\gamma(a),\gamma(b))$.
Following the idea of Marchal, if $\gamma\in AC (x,y,T)$ minimizes $A_h$ over $ AC (x,y)$, then it must minimize $A_h$
over $ AC (x,y,T)$, hence minimizes $A$ over $ AC (x,y,T)$ since $A_h=A+hT$.
Therefore, if $\gamma\in AC (x,y)$ minimizes $A_h$, then $\gamma\vert_{(0,T)}\subset\Omega$.

The Jacobi-Maupertuis metric is defined as $j_h = 2(h+U)g_m$, $g_m$ the usual mass inner product.
A geodesic ray of the $N$-body system  is an arclength parameterized geodesic $\gamma:[0,+\infty) \rightarrow E^N$ such that all of its restrictions to compact subintervals are minimizing geodesics.

\begin{prop}\label{free-time minimizer equivalent to geodesic ray}
    After a suitable reparameterization, a geodesic ray of the metric $j_h$ is a free-time minimizer of $A_h$,
    and a free-time minimizer of $A_h$ is a geodesic ray of the metric.
\end{prop}

Further more, we denote by $d_h$ the Riemannian distance induced by the metric $j_h$,
 Maderna and Venturelli \cite{10.4007/annals.2020.192.2.5} have pointed out that the completion of $(\Omega, d_h)$ is exactly $(E^N,\phi_h)$.
 $\phi_h(h\geq0)$ is a distance,
 it is verified by Maderna \cite{maderna_2012,10.4007/annals.2020.192.2.5}.

Here we mention an important estimate also given by
Maderna \cite{maderna_2012} for $\phi_h(\cdot,\cdot)$ that 
    for a Newtonian $N$-body system in a configuration space $E^{N}$, 
    the masses of the particles and the number $N$ determine two positive constants
    $C_1,C_2$ such that 
    \begin{equation}\label{bound for phi(x,y,t)}
        \phi(x,y,t)\leq C_1\frac{l^2}{t}+C_2\frac{t}{l}
    \end{equation}
    for any $x,y\in E^{N}$, any $l>\|x-y\|$ and $t>0$.
This is needed before applying Ascoli's theorem.

The following result is about a uniform bound for a sequence of action potential, 
we owe it to Maderna and Venturelli \cite[Theorem 2.11]{10.4007/annals.2020.192.2.5} .
\begin{prop}\label{uniform bound for phi_hn}
    If  $\{h_{n}\}$  is  a sequence of bounded non-negative energy constants. There exist constants $\alpha>0$ and $\beta> 0$ such that,  
    \begin{equation}
        \phi_{h_{n}}(x,y)\leq \mu(\|x-y\|),
    \end{equation}
    where $ \mu(\|x-y\|)  = \left(\alpha\|x-y\|+\beta\|x-y\|^2\right)^{1/2}$ and
 $\alpha,\beta$ only depend on the bound of $h_{n}$, the number of bodies $N$ and their masses.
\end{prop}

\begin{proof}
    Suppose $h_n$ are bounded by a positive number $M$, since $\phi_{h_{n}}(x,y)=\inf_{t>0}\left\{\phi(x,y,t)+h_n t\right\}$, 
    and follows from the estimate \ref{bound for phi(x,y,t)} we first estimate the upper bound for $\phi+h_n t$:
    $$
    \phi(x,y,t)+h_n t\leq \frac{C_1l^2}{t}+(\frac{C_2}{l}+h_n)t,
    $$for all $t>0,l>\|x-y\|$,
    now consider the right handside of the above inequality as a function of $t$, we calculate its minimal value to be 
    $2\sqrt{C_1 l^2\left(\frac{C_2}{l}+h_n\right)}=\left( 4C_1 C_2 l +4C_1 h_n l^2 \right)^{1/2}$, thus if we
    let $\alpha=4C_1C_2, \beta = 4C_1 M $, then by $h_n\leq M$ we have 
    $$
    \phi_{h_n}(x,y)\leq \phi(x,y,t)+h_n t \leq \left(\alpha l + \beta l^2\right)^{1/2}
    $$ holds for all $l>\|x-y\|$, 
    then let $l\rightarrow \|x-y\|$ to justify our claim.
\end{proof}

\subsection{Calibrating curves and Busemann functions}

\begin{defi}[Calibrating curves]
    An absolutely continuous curve $\gamma:I \rightarrow E^N$ in some interval $I$ is called a calibrating curve of a function $u\in C(E^N)$ if 
    \begin{equation}
        u\left(\gamma(t_2)\right) - u\left(\gamma(t_1)\right) = A_h\left(\gamma\big\vert_{[t_1,t_2]}\right)
    \end{equation}
    for any closed compact subinterval $[t_1,t_2]\subset I$. We also call $\gamma$ $h$-calibrates $u$ in $I$.
\end{defi}

\begin{defi}[Busemann functions]\label{Def:Busemann functions}
    Let $\gamma:[0,+\infty)\rightarrow\Omega$ be a arbitrary geodesic ray with energy $h\geq0$, the usual Busemann function of \(\gamma\) in the
     length space \((E^N,\phi_h)\) is well-defined as the limit
    \begin{equation}
        b_{\gamma}(x) = \lim_{t\rightarrow+\infty} \left[\phi_h(x,\gamma(t)) - \phi_h(\gamma(0),\gamma(t))\right]
    \end{equation} 
    The normalized Busemann functions are defined by
        \[
\widehat b_\gamma(x)=b_\gamma(0)-b_\gamma(x) = \lim_{t\to+\infty}\left[\phi_h(0,\gamma(t)) - \phi_h(x,\gamma(t))\right],
\]
and are denoted by \(\widehat{\mathcal B}_h\).
\end{defi}

This is the standard Busemann construction for geodesic rays; see, for instance, \cite{Maderna-Venturelli-2026}.

The following proposition is a slight extension of \cite[Theorem 3.2]{10.4007/annals.2020.192.2.5},
where the energy is fixed and positive. In our setting, a class of functions analogous to Busemann functions is defined
through sequences \(h_n\to h\ge 0\), so we record the corresponding statement explicitly.
Since the argument is essentially the same as in \cite{10.4007/annals.2020.192.2.5}, we
defer the proof to Appendix \ref{appendix}.

\begin{prop}\label{Maderna: exist calibrating curve on [0,+infty)}
    For $h_n\geq 0$ and $h_n\rightarrow h\geq 0$, if $u\in C^0(E^N)$ and $u(x)=\lim_{n\rightarrow\infty}\left(\phi_{h_n}(0,p_n) - \phi_{h_n}(x,p_n)\right)$ 
    locally uniformly for some sequence $p_n\in E^{N}$ 
    with $\|p_n\|\rightarrow+\infty$. 
    Then for any initial configuration $x\in E^N$, there exists a calibrating curve $\gamma$  defined over $[0,+\infty)$ with $\gamma(0)=x$.
    i.e.,
    \begin{equation*}
        u(\gamma(t)) - u(x)= A_h \left(\gamma\big\vert_{[0,t]}\right)
    \end{equation*}
    for all $t>0$.
\end{prop}

 \begin{rem}\label{Busemann functions exist}
The functions \(u\) produced by Proposition \ref{Maderna: exist calibrating curve on [0,+infty)}
are called horofunctions, and we denote \(\mathscr{B}_h\) the class of all such functions, \(\mathscr{B}_h\) is 
also called Gromov ideal boundary,  it will be shown that \(\mathscr{B}_h\) is non-empty in the proof of Theorem \ref{thm:compactness_GR_Busemann}.

Recall that each horofunction is a global viscosity solution to the H-J equation \Cref{HJ equation}.
   \end{rem}
   
\begin{lem}\label{lem:Busemann_data_dominated}
Every \(u\in\mathscr{B}_h\) is dominated by \(L+h\), we write \(u\prec L+h\).
When \(h_n\equiv h\), the above construction reduces to the standard horofunction/Busemann-type
functions considered in \cite{10.4007/annals.2020.192.2.5}.
\end{lem}

\begin{proof}
By definition, there exist \(h_n\to h\), \(p_n\in E^N\) with \(\|p_n\|\to+\infty\), and
\[
u_n(x):=\phi_{h_n}(0,p_n)-\phi_{h_n}(x,p_n)
\]
such that \(u_n\to u\) locally uniformly.
For every \(x,y\in E^N\),
\[
u_n(y)-u_n(x)=\phi_{h_n}(x,p_n)-\phi_{h_n}(y,p_n)\le \phi_{h_n}(x,y).
\]
Passing to the limit, and using \(h_n\to h\), we obtain
\[
u(y)-u(x)\le \phi_h(x,y),
\]
hence \(u\prec L+h\).
\end{proof}

\begin{rem}\label{rem:calibrating_implies_hmin}
If \(u\prec L+h\) and \(\gamma\) \(h\)-calibrates \(u\), then \(\gamma\) is an \(h\)-minimizer, i.e.
\[
A_h\!\left(\gamma|_{[a,b]}\right)=\phi_h(\gamma(a),\gamma(b))
\]
for every compact interval \([a,b]\). This is the standard implication recorded in
\cite[Remark 2.14]{10.4007/annals.2020.192.2.5}.
\end{rem}

\begin{thm}[Burgos--Maderna]\label{thm:BM_asymptotics}
Every geodesic ray of the Newtonian \(N\)-body problem with nonnegative energy is expansive. More precisely, if \(\gamma:[0,+\infty)\to\Omega\) is a geodesic ray of energy \(h\ge 0\), then
\[
\gamma(t)=at+O(t^{2/3})\qquad (t\to+\infty)
\]
for some \(a\in E^N\). In particular, the mutual distances along \(\gamma\) grow either with order \(t\) or with order \(t^{2/3}\). If \(h=0\), the motion is completely parabolic; if \(h>0\), it is hyperbolic or partially hyperbolic.
\end{thm}

\section{Compactness of geodesic-ray data and normalized Busemann functions}\label{sec:compactness}

We give the proof of \Cref{thm:compactness_GR_Busemann} and of which the results in this section.
before that, 
we need some vital technic in the following lemma.

\begin{lem}\label{lem:limit_not_collision}
Let \((x_n,v_n)\subset GR\) and assume that
\[
(x_n,v_n)\to (x_0,v_0)\qquad\text{in }TE^N.
\]
Then \(x_0\in\Omega\).
\end{lem}

\begin{proof}
For each \(n\), let \(\gamma_n:[0,+\infty)\to\Omega\) be the geodesic ray with
\[
\gamma_n(0)=x_n,\qquad \dot\gamma_n(0)=v_n.
\]
By Theorem \ref{thm:BM_asymptotics}, the total energy
\[
h_n=\frac12\|v_n\|^2-U(x_n)
\]
of \(\gamma_n\) satisfies \(h_n\ge 0\).
Therefore
\[
U(x_n)=\frac12\|v_n\|^2-h_n\le \frac12\|v_n\|^2.
\]
Since \(v_n\to v_0\), the sequence \(\{U(x_n)\}\) is bounded.
If \(x_0\in\Delta\), then \(U(x_n)\to+\infty\), a contradiction.
Hence \(x_0\in\Omega\).
\end{proof}

Now we can conduct the proof of Theorem \ref{thm:compactness_GR_Busemann}

\begin{proof}[Proof of Theorem \ref{thm:compactness_GR_Busemann}]
By Lemma \ref{lem:limit_not_collision}, we have \(x_0\in\Omega\).
Hence
\[
h_n=\frac12\|v_n\|^2-U(x_n)\to \frac12\|v_0\|^2-U(x_0)=:h,
\]
and since each \(h_n\ge 0\), it follows that \(h\ge 0\).

For each \(n\), Theorem \ref{thm:BM_asymptotics} gives
\[
\gamma_n(t)=a_n t+O(t^{2/3}) \qquad \text{as } t\to+\infty,
\]
hence \(\|\gamma_n(t)\|\to+\infty\) as \(t\to+\infty\).
Choose \(t_n\ge n\) such that
\[
p_n:=\gamma_n(t_n),\qquad \|p_n\|\ge n.
\]
Then \(\|p_n\|\to+\infty\).

Define
\[
u_n(x):=\phi_{h_n}(0,p_n)-\phi_{h_n}(x,p_n).
\]
By Proposition \ref{uniform bound for phi_hn},
\[
|u_n(x)-u_n(y)|\le \phi_{h_n}(x,y)\le \mu(\|x-y\|)
\]
for all \(x,y\in E^N\), where \(\mu\) depends only on a common bound for \(\{h_n\}\).
Thus \(\{u_n\}\) is equicontinuous and locally bounded on \(E^N\).
By Ascoli's theorem, after passing to a subsequence, we may assume that
\[
u_n\to u
\qquad \text{locally uniformly on }E^N.
\]
Since \(h_n\to h\) and \(\|p_n\|\to+\infty\), the locally uniform limit \(u\) belongs to
\(\mathscr{B}_h\) by definition and hence $u$ is dominated by $L+h$ according to Lemma \ref{lem:Busemann_data_dominated}.

Let \(\gamma:[0,t_*)\to\Omega\) be the maximal classical solution with initial datum
\[
\gamma(0)=x_0,\qquad \dot\gamma(0)=v_0.
\]
Fix \(T<t_*\).
By continuous dependence on initial data,
\[
\gamma_n\to\gamma,\qquad \dot\gamma_n\to\dot\gamma
\]
uniformly on \([0,T]\).

Since \(\gamma([0,T])\) is compact and contained in \(\Omega\), for all sufficiently large \(n\),
also \(\gamma_n([0,T])\) stays in a compact subset of \(\Omega\).
Therefore, for every \(t\in[0,T]\),
\[
A_{h_n}\!\left(\gamma_n\big|_{[0,t]}\right)\to
A_h\!\left(\gamma\big|_{[0,t]}\right).
\]

On the other hand, \(\gamma_n\) \(h_n\)-calibrates \(u_n\), so
\[
u_n(\gamma_n(t))-u_n(x_n)=A_{h_n}\!\left(\gamma_n\big|_{[0,t]}\right),
\qquad t\in[0,T].
\]
Passing to the limit, using \(x_n\to x_0\), \(\gamma_n(t)\to\gamma(t)\), and the local uniform
convergence \(u_n\to u\), we obtain
\[
u(\gamma(t))-u(x_0)=A_h\!\left(\gamma\big|_{[0,t]}\right),
\qquad t\in[0,T].
\]
Since \(T<t_*\) is arbitrary, \(\gamma\) \(h\)-calibrates \(u\) on \([0,t_*)\).

We now prove that \(t_*=+\infty\).
Assume by contradiction that \(t_*<+\infty\), and fix \(t'\in(0,t_*)\).
Since \(u\in\mathscr{B}_h\), Proposition \ref{Maderna: exist calibrating curve on [0,+infty)}
provides an  curve \(h\)-calibrating $u$
\[
\rho:[0,+\infty)\to E^N,\qquad \rho(0)=\gamma(t').
\]
Concatenating \(\gamma|_{[0,t']}\) with \(\rho\), we obtain a global \(h\)-calibrating curve
\(\widetilde\gamma\).
By Remark \ref{rem:calibrating_implies_hmin}, \(\widetilde\gamma\) is a free-time minimizer.
Since \(\gamma(t')\in\Omega\), Marchal's theorem implies that \(\widetilde\gamma\) is a classical solution
near \(t'\), hence there is no corner at \(t'\), and therefore
\[
\dot\rho(0)=\dot\gamma(t').
\]
By uniqueness of classical solutions, \(\rho(s)=\gamma(t'+s)\) as long as both curves are defined,
which extends \(\gamma\) beyond \(t_*\), a contradiction.
Thus \(t_*=+\infty\).

Finally, since $u$ is dominated by $L+h$,
By Remark \ref{rem:calibrating_implies_hmin}, \(\gamma\) is a free-time minimizer of
\(A_h\), and by Proposition \ref{free-time minimizer equivalent to geodesic ray}, \(\gamma\) is a geodesic ray.
Hence \((x_0,v_0)\in GR\).

This proves all the assertions of Theorem \ref{thm:compactness_GR_Busemann}.

\end{proof}

The preceding theorem is a stability statement for calibrated motions. 
It shows that the class of initial data selected by the Jacobi--Maupertuis variational principle is closed under ambient 
convergence and that no collision can appear in the limit. We now record a companion compactness statement for the boundary 
data at infinity. This is the Hamilton--Jacobi counterpart of the previous dynamical compactness: along a convergent sequence 
of geodesic-ray initial data, the exact normalized Busemann functions form a locally compact family, and every cluster point is 
still a normalized horofunction calibrating the limiting ray, and more importantly, is a viscosity solution of the limiting Hamilton--Jacobi
 equation.

For a geodesic ray \(\gamma:[0,+\infty)\to\Omega\) of energy \(h\ge 0\), we denote by
\[
\widehat b_\gamma(x)
:=
\lim_{t\to+\infty}\Bigl(\phi_h(0,\gamma(t))-\phi_h(x,\gamma(t))\Bigr)
\]
its normalized Busemann function in the sign convention used in this paper. Equivalently,
\[
\widehat b_\gamma(x)=b_\gamma(0)-b_\gamma(x),
\]
where \(b_\gamma\) denotes the standard Busemann function.

\begin{prop}\label{prop:precompact_busemann_data}
Let \((x_n,v_n)\subset GR\) and assume that
\[
(x_n,v_n)\to (x_0,v_0)\qquad\text{in }TE^N.
\]
For each \(n\), let \(\gamma_n:[0,+\infty)\to\Omega\) be the corresponding geodesic ray, and let
\[
h_n=\frac12\|v_n\|^2-U(x_n).
\]
Then the family of normalized Busemann functions \(\{\widehat b_{\gamma_n}\}\) is relatively
compact in \(C_{\mathrm{loc}}(E^N)\).
 
More precisely, after passing to a subsequence, there exists a function
\[
u\in \mathscr{B}_h,
\qquad
h=\frac12\|v_0\|^2-U(x_0),
\]
such that
\[
\widehat b_{\gamma_n}\to u
\qquad\text{locally uniformly on }E^N.
\]
Moreover, if \(\gamma\) is the geodesic ray with initial datum \((x_0,v_0)\), then \(\gamma\)
\(h\)-calibrates \(u\).
\end{prop}

\begin{proof}
By Theorem \ref{thm:compactness_GR_Busemann}, we already know that
\[
x_0\in\Omega,\qquad h_n\to h:=\frac12\|v_0\|^2-U(x_0)\ge 0,
\]
and that the classical solution with initial datum \((x_0,v_0)\) is a geodesic ray, which we
denote by \(\gamma\).

For every \(n\), the function \(\widehat b_{\gamma_n}\) satisfies
\[
\widehat b_{\gamma_n}(0)=0,
\]
and for all \(x,y\in E^N\),
\[
\bigl|\widehat b_{\gamma_n}(y)-\widehat b_{\gamma_n}(x)\bigr|
\le \phi_{h_n}(x,y).
\]
Since \((h_n)\) is bounded, Proposition \ref{uniform bound for phi_hn} provides a common modulus
of continuity. Hence the family \(\{\widehat b_{\gamma_n}\}\) is locally bounded and equicontinuous.
By Ascoli's theorem, after passing to a subsequence, there exists
\[
u\in C^0(E^N)
\]
such that
\[
\widehat b_{\gamma_n}\to u
\qquad\text{locally uniformly on }E^N.
\]

We now prove that \(u\in\mathscr{B}_h\).
For each \(n\), the convergence
\[
\phi_{h_n}(0,\gamma_n(t))-\phi_{h_n}(x,\gamma_n(t))
\longrightarrow
\widehat b_{\gamma_n}(x)
\qquad (t\to+\infty)
\]
is locally uniform in \(x\). Therefore we can choose \(t_n\ge n\) such that
\[
\sup_{\|x\|\le n}
\left|
\widehat b_{\gamma_n}(x)
-
\Bigl(\phi_{h_n}(0,\gamma_n(t_n))-\phi_{h_n}(x,\gamma_n(t_n))\Bigr)
\right|
\le \frac1n.
\]
Set
\[
p_n:=\gamma_n(t_n).
\]
Since each \(\gamma_n\) is a geodesic ray, it is expansive, hence \(\|p_n\|\to+\infty\).
Now let \(K\subset E^N\) be compact. For \(n\) large enough, \(K\subset B(0,n)\), and thus
\[
\sup_{x\in K}
\left|
\Bigl(\phi_{h_n}(0,p_n)-\phi_{h_n}(x,p_n)\Bigr)-u(x)
\right|
\le
\frac1n
+
\sup_{x\in K}|\widehat b_{\gamma_n}(x)-u(x)|.
\]
Letting \(n\to\infty\), we obtain
\[
u(x)=\lim_{n\to\infty}\Bigl(\phi_{h_n}(0,p_n)-\phi_{h_n}(x,p_n)\Bigr)
\]
locally uniformly on \(E^N\). Hence \(u\in\mathscr{B}_h\).

It remains to prove that \(\gamma\) \(h\)-calibrates \(u\). For each \(n\) and every \(t\ge 0\),
the geodesic-ray property of \(\gamma_n\) implies
\[
\phi_{h_n}(x_n,\gamma_n(s))
=
\phi_{h_n}(x_n,\gamma_n(t))
+
\phi_{h_n}(\gamma_n(t),\gamma_n(s))
\qquad \text{for all }s>t.
\]
Passing to the limit as \(s\to+\infty\), we obtain
\[
\widehat b_{\gamma_n}(\gamma_n(t))-\widehat b_{\gamma_n}(x_n)
=
\phi_{h_n}(x_n,\gamma_n(t))
=
A_{h_n}\!\left(\gamma_n\big|_{[0,t]}\right).
\]
Fix \(T>0\). Since \((x_n,v_n)\to(x_0,v_0)\) and \(x_0\in\Omega\), continuous dependence on
initial data gives
\[
\gamma_n\to\gamma,\qquad \dot\gamma_n\to\dot\gamma
\]
uniformly on \([0,T]\). As in the proof of Theorem \ref{thm:compactness_GR_Busemann}, it
follows that
\[
A_{h_n}\!\left(\gamma_n\big|_{[0,t]}\right)
\to
A_h\!\left(\gamma\big|_{[0,t]}\right)
\qquad\text{for every }t\in[0,T].
\]
Passing to the limit in
\[
\widehat b_{\gamma_n}(\gamma_n(t))-\widehat b_{\gamma_n}(x_n)
=
A_{h_n}\!\left(\gamma_n\big|_{[0,t]}\right),
\]
we get
\[
u(\gamma(t))-u(x_0)=A_h\!\left(\gamma\big|_{[0,t]}\right)
\qquad\text{for all }t\in[0,T].
\]
Since \(T>0\) is arbitrary, \(\gamma\) \(h\)-calibrates \(u\).
\end{proof}

\begin{rem}\label{rem:widehatB_vs_tildeB}
Proposition \ref{prop:precompact_busemann_data} does not assert that the collection of exact normalized Busemann functions is
 closed in \(C_{\mathrm{loc}}(E^N)\). What is proved is the natural precompactness statement: every sequence of exact
  normalized Busemann functions associated with convergent geodesic-ray data has a locally uniformly convergent subsequence, 
  and each cluster point belongs to the horofunction class \(\mathscr B_h\). Thus \(\mathscr B_h\) can be viewed as a convenient closure class for normalized Busemann data.
\end{rem}

The compactness argument Theorem \ref{thm:compactness_GR_Busemann} 
and Proposition \ref{prop:precompact_busemann_data} has  two topological 
consequences. First, the full geodesic-ray set is closed in the ambient phase space. 
Second, the limit horofunction obtained by $-u = -\lim_{n\to\infty}\widehat{b}_{\gamma_n}$ is
a viscosity solution of $H=h$.
Third, after fixing a collision-free hyperbolic limit shape,
 the corresponding slice is also closed. The latter statement uses the uniqueness of 
 the fixed-shape Busemann function proved by Maderna--Venturelli.

\begin{cor}\label{cor:GR_closed}
The set \(GR\) is closed in the ambient space \(TE^N\).
\end{cor}

\begin{proof}
This is immediate from Theorem \ref{thm:compactness_GR_Busemann}.
\end{proof}

\begin{cor}[Hamilton--Jacobi interpretation of the limiting horofunction]
\label{cor:HJ_interpretation_horofunction}
Let \((x_n,v_n)\subset GR\), $h_n, h, \widehat b_{\gamma_n}$ and its limit $u$ be as in Proposition
\ref{prop:precompact_busemann_data}, then
\(-\widehat b_{\gamma_n}\) and $-u$
are viscosity solutions of
\[
H(x,Du)=h_n~\text{  and }~ H(x,Du)=h
\]
on \(\Omega\) respectively.

\end{cor}

\begin{proof}
Set
\[
w_n:=-\widehat b_{\gamma_n}.
\]
Then
\[
w_n(x)=
\lim_{t\to+\infty}
\bigl(\phi_{h_n}(x,\gamma_n(t))-\phi_{h_n}(0,\gamma_n(t))\bigr),
\]
which is the standard normalized Busemann function associated with the geodesic ray \(\gamma_n\).

We first prove the subsolution property. For all \(x,y\in E^N\), the triangle inequality gives
\[
\phi_{h_n}(y,\gamma_n(t))-\phi_{h_n}(x,\gamma_n(t))
\le \phi_{h_n}(y,x)=\phi_{h_n}(x,y).
\]
Letting \(t\to+\infty\), we obtain
\[
w_n(y)-w_n(x)\le \phi_{h_n}(x,y).
\]
Thus \(w_n\prec L+h_n\), and hence \(w_n\) is a viscosity subsolution of \(H(x,Dw)=h_n\) by the standard domination criterion for the \(N\)-body Hamilton--Jacobi equation \cite[Proposition~2.4]{10.4007/annals.2020.192.2.5}.

We next prove the supersolution property. Since \(\widehat b_{\gamma_n}\) is a normalized Busemann function, it belongs to the horofunction class considered in Proposition \ref{Maderna: exist calibrating curve on [0,+infty)}. Hence, for every \(x\in E^N\), there exists a forward curve \(\eta:[0,+\infty)\to E^N\), \(\eta(0)=x\), which calibrates \(\widehat b_{\gamma_n}\), namely
\[
\widehat b_{\gamma_n}(\eta(t))-\widehat b_{\gamma_n}(x)
=
A_{h_n}(\eta|_{[0,t]}).
\]
Define the reversed curve \(\sigma:(-\infty,0]\to E^N\) by \(\sigma(s)=\eta(-s)\). Since \(L(x,v)=L(x,-v)\), the preceding identity becomes
\[
w_n(\sigma(0))-w_n(\sigma(-t))
=
A_{h_n}(\sigma|_{[-t,0]}),
\qquad t\ge0.
\]
Thus \(w_n\) has a backward calibrating curve through every point. The usual domination--calibration argument, equivalently the supersolution criterion for calibrated dominated functions, implies that \(w_n\) is a viscosity supersolution of \(H(x,Dw)=h_n\); see \cite[Section~2.2.1]{10.4007/annals.2020.192.2.5}. Therefore \(-\widehat b_{\gamma_n}=w_n\) is a viscosity solution of \(H(x,Dw)=h_n\) on \(\Omega\).

Finally, if \(\widehat b_{\gamma_n}\to u\) locally uniformly and \(h_n\to h\), then \(w_n=-\widehat b_{\gamma_n}\to -u\) locally uniformly. Since the Hamiltonians
\[
F_n(x,p)=H(x,p)-h_n
\]
converge locally uniformly on \(\Omega\times(E^N)^*\) to \(F(x,p)=H(x,p)-h\), the stability theorem for viscosity solutions yields that \(-u\) is a viscosity solution of \(H(x,Dw)=h\) on \(\Omega\).
\end{proof}

\begin{cor}[Closedness of the fixed-shape slice]\label{cor:GRa_closed}
Let \(a\in\Omega\), and define
\[
GR(a):=\{(x,v)\in T\Omega:\ \text{the hyperbolic solution starting from }(x,v)
\text{ with limit shape }a\}.
\]
Then \(GR(a)\) is closed in the ambient space \(TE^N\).
\end{cor}

\begin{proof}
Let \((x_n,v_n)\subset GR(a)\) and assume that
\[
(x_n,v_n)\to (x_0,v_0)\qquad\text{in }TE^N.
\]
For each \(n\), let \(\gamma_n\) be the corresponding hyperbolic geodesic ray.

By Proposition \ref{prop:precompact_busemann_data}, after passing to a subsequence, the
functions \(\widehat b_{\gamma_n}\) converge locally uniformly to some function \(u\), and the
geodesic ray \(\gamma\) with initial datum \((x_0,v_0)\) \(h\)-calibrates \(u\).

On the other hand, Based on \cite[Theorem 2.1 and Lemma 3.1]{Maderna-Venturelli-2026},
\(\{\widehat b_{\gamma_n}=\widehat b_a\}\) are the same and the normalized hyperbolic viscosity solution with limit shape $a$ and so is $u$.

Therefore according to Maderna--Venturelli\cite[Definition 2.3, Remark 2.5]{Maderna-Venturelli-2026}, that
every calibrating curve of \(u\) is a hyperbolic geodesic
ray with limit shape \(a\). Consequently, \(\gamma\) is a hyperbolic geodesic ray with limit
shape \(a\), that is,
\[
(x_0,v_0)\in GR(a).
\]
Thus \(GR(a)\) is closed in \(TE^N\).
\end{proof}

\begin{rem}

By Maderna--Venturelli's fixed-shape uniqueness theorem \cite{Maderna-Venturelli-2026}, 
\(\widehat b_{\gamma_n}\) is nothing but
the
well-defined normalized Busemann function associated with \(a\), which can be denoted by
\(-b_a\)
 in the sign convention of \cite{Maderna-Venturelli-2026}.

\end{rem}

\section{Geometric measure structure of the fixed-shape slice}\label{sec:geometric}

The compactness results proved above were formulated in the full configuration space \(E^N\) and in the ambient phase space \(TE^N\). In this section, in order to state dimension and rectifiability properties in the same framework as Berti--Polimeni--Terracini \cite{Terracini-2025}, 
we pass to the center-of-mass frame and work on the reduced configuration space $X$ with dimension $d(N-1)$.
\[
X:=\left\{x=(r_1,\dots,r_N)\in E^N:\ \sum_{i=1}^N m_i r_i=0\right\}.
\]
We write
\[
\Omega_X:=\Omega\cap X,
\qquad
\Delta_X:=X\setminus\Omega_X.
\]

Fix a collision-free configuration \(a\in\Omega_X\) and let
\[
h:=\frac12\|a\|^2.
\]
We have denoted by \(GR_X(a)\) the fixed-shape slice in the reduced phase space. 
This is the reduced analogue of the fixed-shape slice 
considered in the full space. By Maderna--Venturelli's 
uniqueness theorem, the collision-free configuration \(a\) 
determines a canonical Busemann function \(b_a\), which is 
a hyperbolic viscosity solution of the stationary Hamilton--Jacobi 
equation at energy \(h\), see \cite{Maderna-Venturelli-2026}. 
The set \(GR_X(a)\) may therefore be interpreted as the 
phase-space slice of initial data whose corresponding 
rays are calibrated by this fixed-shape Hamilton--Jacobi 
boundary datum.

Following \cite{Maderna-Venturelli-2026}, for \(\alpha\in(0,1)\) and \(r>0\) we set
\[
C_a(\alpha)=\{x\in X:\ \langle x,a\rangle\ge \alpha\|x\|\|a\|\},
\qquad
C_a(\alpha,r)=C_a(\alpha)\setminus B(r),
\]
and we use the notation
\[
TC_a(\alpha,r):=C_a(\alpha,r)\times X.
\]
The following proposition is a geometric reformulation of the theorem of the cone together with the \(C^1\)-regularity of the limit-shape map proved by Maderna--Venturelli.

\begin{prop}[Local graph structure on a cutted cone]\label{prop:GRa_local_graph}
There exist \(\alpha\in(\alpha_0(a),1)\), \(r>0\), and a \(C^1\) map
\[
V_a:C_a(\alpha,r)\longrightarrow X
\]
such that
\[
GR_X(a)\cap TC_a(\alpha,r)
=
\Graph(V_a|_{C_a(\alpha,r)})
:=
\{(x,V_a(x)):\ x\in C_a(\alpha,r)\}.
\]
In particular, this local piece of \(GR_X(a)\) is an \(d(N-1)\)-dimensional \(C^1\) embedded submanifold of \(TX\).
\end{prop}

\begin{proof}
By the theorem of the cone, for every \(\alpha\in(\alpha_0(a),1)\) there exists \(r_0>0\) such that, for each \(x\in C_a(\alpha,r_0)\), there is a unique hyperbolic geodesic ray starting at \(x\) and having limit shape \(a\); see \cite[Theorem~3.1]{Maderna-Venturelli-2026}.

On the other hand, in the proof of Proposition~5.1 of \cite{Maderna-Venturelli-2026}, using the \(C^1\)-regularity of the limit-shape map established in Section~4.3 of that paper, the authors construct a \(C^1\) map
\[
V_a:C_a(\alpha,r_1)\to X
\]
for some \(r_1\ge r_0\), such that for every \(x\in C_a(\alpha,r_1)\), the hyperbolic motion with initial datum \((x,V_a(x))\) has limit shape \(a\), remains in the cone, and is the unique hyperbolic motion with these properties. Increasing the radius if necessary, and applying the theorem of the cone, this motion is the unique hyperbolic geodesic ray with limit shape \(a\) starting from \(x\). Hence
\[
GR_X(a)\cap TC_a(\alpha,r_1)=\Graph(V_a|_{C_a(\alpha,r_1)}).
\]
Renaming \(r_1\) as \(r\) gives the asserted identity.

Finally, the graph map
\[
\Psi:C_a(\alpha,r)\to TX,
\qquad
\Psi(x):=(x,V_a(x)),
\]
is of class \(C^1\). Its differential is injective at every point because the first component is the identity. Therefore \(\Graph(V_a|_{C_a(\alpha,r)})\) is an \(d(N-1)\)-dimensional \(C^1\) embedded submanifold of \(TX\).
\end{proof}

\begin{thm}[Rectifiability and Hausdorff dimension of the fixed-shape calibrated slice]\label{thm:GRa_geom_{d(N-1)}easure}
The fixed-shape slice \(GR_X(a)\) is countably \(d(N-1)\)-rectifiable in \(TX\). Moreover,
\[
\dim_H GR_X(a)= d(N-1).
\]
More precisely, for every compact set \(K\Subset C_a(\alpha,r)\) with \(\mathcal L^{d(N-1)}(K)>0\),
\[
0<\mathcal H^{d(N-1)}\bigl(\Graph(V_a|_K)\bigr)<+\infty.
\]
\end{thm}

\begin{proof}
Set
\[
G_a:=\Graph(V_a|_{C_a(\alpha,r)})\subset TX,
\]
where \(\alpha,r\) and \(V_a\) are given by Proposition~\ref{prop:GRa_local_graph}. Let \(\Phi_t\) denote the Hamiltonian flow on \(TX\).

\smallskip
\noindent\textbf{Step 1. Countable rectifiability.}
We claim that
\[
GR_X(a)\subset \bigcup_{n=0}^{\infty}\Phi_{-n}(G_a).
\]
Indeed, let \((x,v)\in GR_X(a)\), and let \(\gamma:[0,+\infty)\to\Omega_X\) be the corresponding
 hyperbolic geodesic ray. 
 
Since \(\gamma\) has limit shape \(a\), we have
\[
\gamma(t)=ta+o(t)\qquad\text{as }t\to+\infty,
\]
then

\[
\lim_{t\to+\infty} \left\langle\frac{\gamma(t)}{\|\gamma(t)\|}, \frac{a}{\|a\|}\right\rangle  = 1
\]
On the otherhand, since $\alpha\in(\alpha_0(a),1)$, there exists  $T>0$ such that $\langle\gamma(t),a\rangle\geq\alpha\|\gamma(t)\|\|a\|$ for all $t\ge T$.
Hence 
\[
\gamma(t)\in C_a(\alpha,r)
\qquad\text{for all }t\ge T.
\]

Choose an integer \(n\ge T\). Then \(\gamma|_{[n,+\infty)}\) is a hyperbolic geodesic ray starting from \(\gamma(n)\in C_a(\alpha,r)\) and having limit shape \(a\). By Proposition~\ref{prop:GRa_local_graph}, it must coincide with the unique ray issued from \(\gamma(n)\), so
\[
\dot\gamma(n)=V_a(\gamma(n)).
\]
Thus
\[
\Phi_n(x,v)=(\gamma(n),\dot\gamma(n))=(\gamma(n),V_a(\gamma(n)))\in G_a,
\]
and the claim follows.

For each \(n\in\mathbb N\cup\{0\}\), let
\[
U_n:=\{x\in C_a(\alpha,r):\ \Phi_{-n}(x,V_a(x))\ \text{is nonempty}\},
\]
and define
\[
F_n:U_n\to TX,
\qquad
F_n(x):=\Phi_{-n}(x,V_a(x)).
\]
Since \(\Phi_{-n}\) is smooth on its domain of definition and \(V_a\) is \(C^1\), each \(F_n\) is \(C^1\). 
Choose a linear isomorphism \(L:\mathbb R^{d(N-1)}\to X\), set \(E_n:=L^{-1}(U_n)\), 
and write \(f_n:=F_n\circ L:E_n\to TX\). For every \(\ell\in\mathbb N\), put compact sets
\[
K_{n,\ell}:=\left\{z\in E_n:\ |z|\le \ell,\ \operatorname{dist}(z,\mathbb R^{d(N-1)}\setminus E_n)\ge \frac1\ell\right\}.
\]
Then \(E_n=\bigcup_{\ell=1}^\infty K_{n,\ell}\), and each restriction \(f_n|_{K_{n,\ell}}\) 
is Lipschitz because \(f_n\) is \(C^1\). Hence
\[
F_n(U_n)\subset \bigcup_{\ell=1}^{\infty} f_n(K_{n,\ell}),
\]
and therefore
\[
GR_X(a)
\subset
\bigcup_{n=0}^{\infty}\bigcup_{\ell=1}^{\infty} f_n(K_{n,\ell}).
\]
Thus \(GR_X(a)\) is contained in a countable union of Lipschitz images of subsets of \(\mathbb R^{d(N-1)}\), and hence is countably \(d(N-1)\)-rectifiable.

\smallskip
\noindent\textbf{Step 2. Positive and finite local \(d(N-1)\)-measure on graph patches.}
Let
\[
\Psi:C_a(\alpha,r)\to TX,
\qquad
\Psi(x):=(x,V_a(x)).
\]
For a compact set \(K\Subset C_a(\alpha,r)\), the area formula gives
\[
\mathcal H^{d(N-1)}(\Psi(K))
=
\int_K J_{d(N-1)}\Psi(x)\,d\mathcal L^{d(N-1)}(x),
\]
where
\[
J_{d(N-1)}\Psi(x)=\sqrt{\det\!\bigl(I+(DV_a(x))^\ast DV_a(x)\bigr)}.
\]
Since \((DV_a(x))^\ast DV_a(x)\) is positive semidefinite, \(J_{d(N-1)}\Psi(x)\ge 1\). Hence, if \(\mathcal L^{d(N-1)}(K)>0\),
\[
\mathcal H^{d(N-1)}\bigl(\Graph(V_a|_K)\bigr)=\mathcal H^{d(N-1)}(\Psi(K))\ge \mathcal L^{d(N-1)}(K)>0.
\]
On the other hand, \(DV_a\) is continuous, so \(J_{d(N-1)}\Psi\) is bounded on compact subsets; therefore \(\mathcal H^{d(N-1)}(\Graph(V_a|_K))<+\infty\).

\smallskip
\noindent\textbf{Step 3. Hausdorff dimension.}
Since \(GR_X(a)\) is countably \(d(N-1)\)-rectifiable, it is contained in a countable union of Lipschitz images of subsets of \(\mathbb R^{d(N-1)}\). Hence
\[
\dim_H GR_X(a)\le d(N-1).
\]
Conversely, choose a compact set \(K\Subset C_a(\alpha,r)\) with \(\mathcal L^{d(N-1)}(K)>0\). By Step~2,
\[
0<\mathcal H^{d(N-1)}\bigl(\Graph(V_a|_K)\bigr),
\qquad
\Graph(V_a|_K)\subset GR_X(a).
\]
Therefore \(\dim_H GR_X(a)\ge d(N-1)\). Thus
\[
\dim_H GR_X(a)= d(N-1).
\]
\end{proof}

\begin{rem}[Relation with singular-set rectifiability]
Theorem~\ref{thm:GRa_geom_measure} concerns a  phase-space object: the fixed-shape slice of 
initial data whose corresponding rays are selected by the canonical Busemann solution \(b_a\).
 While for the singular-set rectifiability theorem of Berti--Polimeni--Terracini \cite{Terracini-2025}, the object is the irregular set of a Hamilton--Jacobi value function in configuration space and has codimension one. In the present setting, the slice \(GR_X(a)\) is locally a \(C^1\) graph over configuration space and is therefore an \(d(N-1)\)-dimensional rectifiable object in phase space.
\end{rem}

\begin{rem}[Higher-dimensional measures]
Since \(GR_X(a)\) is countably \(d(N-1)\)-rectifiable, one has \(\mathcal H^q(GR_X(a))=0\) for every \(q> d(N-1)\). 
In particular, if \(\Sigma_h:=\{(x,v)\in TX:\ H(x,v)=h\}\), then \(\mathcal H^{\dim\Sigma_h}(GR_X(a))=0\) since $d\geq2$ and \(\dim \Sigma_h=2d(N-1)-1\). 
Thus \(GR_X(a)\) has positive size in its natural dimension \(d(N-1)\), but zero surface measure in the higher-dimensional energy shell.
\end{rem}

\section*{Appendix}\label{appendix}

Before proving the proposition \ref{Maderna: exist calibrating curve on [0,+infty)}, we need 
the famous von Zeipel's theorem, 
that
 if the solution $x(t)$ has a singularity at $t^*$, 
 then $\lim_{t\rightarrow t^*} I(x(t)) = I^*$ exists and lies in $[0, +\infty]$.
 In particular, in the case of  $\lim_{t\rightarrow t^*}I(x(t)) = I^*<+\infty$, 
 there exists a $x^* \in \Delta$ such that $\lim_{t \rightarrow t^*}x(t) = x^*$, 
 i.e., $x(t)$ suffers a collision.
 In the case of $I^* = +\infty$, the particles of $x(t)$ are said to undergo a pseudocollision.
 See von Zeipel \cite{von1908singularites} and McGehee \cite{zbMATH04009869} for more detail.

\begin{proof}[Proof of proposition \ref{Maderna: exist calibrating curve on [0,+infty)}]
    We first show that for any $r>0$ there exist some $y_r\in E^N$ and a curve $\gamma_r\in AC (x,y_r)$ defined in a finite interval, 
    such that $\|x-y_r\|=r$ and $u(y_r) - u(x)= A_h (\gamma_r)$.
    
    Denote $\phi_{h_n}(\cdot,p_n)$ by $v_n(\cdot)$ as a function  in $E^N$, then it is continuous since $\phi_{h_n}$ is a distance.
    According to Proposition \ref{free-time minimizer of phi_h exist}, the action $\phi_{h_n}(x,p_n)$ attains its minimum in $ AC (x,p_n)$ 
    at some $\gamma_n$ defined over $[0,t_n]$ with $\gamma_n(t_n) = p_n$, then for any $t\in[0, t_n]$,
    \begin{equation*}
        v_n(\gamma(t))-v_n(p_n) = \phi_{h_n}(\gamma(t),p_n)= A_{h_n} \left(\gamma_n\vert_{[t,t_n]}\right),
    \end{equation*}
    which implies that for any subinterval $[t_1, t_2]\subset [0, t_n)$, 
    \begin{align*}
        &v_n(\gamma(t_1)) - v_n(\gamma(t_2)) \\
        = &\big(v_n(p_n)-v_n(\gamma(t_2))\big) - \big(v_n(p_n)-v_n(\gamma(t_1))\big)    \\
        = & -A_{h_n} \left(\gamma\vert_{[t_2,t_n]}\right) + A_{h_n} \left(\gamma\vert_{[t_1,t_n]}\right)\\
        = & A_{h_n} \left(\gamma\vert_{[t_1,t_2]}\right).
    \end{align*}
    Now because $\|p_n\|\rightarrow+\infty$, suppose $\|p_n-x\|>r$, then for each $n$, there is a $y_n=\gamma(\tau_n)$ for some $\tau_n$  such that $\|y_n-x\|=r$, 
    thus we assume $y_n\rightarrow y_r$ for some  $y_r\in\partial B(x,r)$.
    Since $\gamma_n$  is a calibrating curve, the following equation holds:
    \begin{equation*}
     v_n(x) - v_n(y_n)=  A_{h_n} \left(\gamma_n\big\vert_{[0,\tau_n]}\right) =\phi_{h_n}(x,y_n).
    \end{equation*}   
    On the other hand, since $\phi_{h_n}$ is distance, 
    then $\phi_{h_n}(y_n,p_n)\leq\phi_{h_n}(y_n,y_r)+\phi_{h_n}(y_r,p_n)$ and $\phi_{h_n}(y_r,p_n)\leq\phi_{h_n}(p_n,y_n)+\phi_{h_n}(y_n,y_r)$
    implies $\vert v_n(y_n)-v_n(y_r)\vert\leq\phi_{h_n}(y_r,y_n)$, then
    we apply Proposition \ref{uniform bound for phi_hn} to have $\vert v_n(y_n)-v_n(y_r)\vert \leq\mu(\|y_n-y_r\|)\rightarrow 0$ as $n\rightarrow\infty$.
 
    On the other hand, since $\phi_{h_n},\phi_h$ are distances in $E^N$ induced by conformal metrics and $h_n\rightarrow h$, thus $\lim_n\phi_{h_n}(x,y_r) = \phi_h(x,y_r)$,
    and by Proposition \ref{uniform bound for phi_hn}, we have 
    \begin{equation}
        |\phi_{h_n}(x,y_r) - \phi_{h_n}(x,y_n)| \leq\phi_{h_n}(y_r,y_n) \leq\mu\|y_r-y_n\|\rightarrow 0,
    \end{equation}    
    and then
    \begin{equation}
        |\phi_h(x,y_r) - \phi_{h_n}(x,y_n)|\leq|\phi_h(x,y_r) -\phi_{h_n}(x,y_r)| + |\phi_{h_n}(x,y_r) - \phi_{h_n}(x,y_n)|\rightarrow 0,
    \end{equation}
   i.e., $\lim_n\phi_{h_n}(x,y_n) = \phi_h(x,y_r)$.
    Therefore
 \begin{align*}
     u(y_r) - u(x)&=\lim_n\bigl(v_n(x) - v_n(y_r)\bigr)\\
              &=\lim_n\bigl(v_n(x) - v_n(y_n)+v_n(y_n) - v_n(y_r)\bigr)\\
              &=\lim_n\bigl(\phi_{h_n}(x,y_n)+v_n(y_n)-v_n(y_r)\bigr)\\
              &=\phi_h(x, y_r),
 \end{align*}
 we again apply Proposition \ref{free-time minimizer of phi_h exist} to take a minimizer $\gamma_r\in AC (x,y_r)$ 
 such that $ A_h (\gamma_r) = \phi_h(x,y_r) = u(y_r) - u(x)$.
 
 Now we apply Zorn's Lemma to obtain a maximal calibrating curve $\gamma: [0,t^*)\rightarrow E^N$, then we try to prove $t^* = +\infty$.
 We argue by contradiction. Let $t^*\neq+\infty$, $\gamma$ is a free-time minimizer thus it is a true motion in $[0, t^*)$, 
 the maximal property of $t^*$ implies it is a singularity. 
 Applying von Zeipel's theorem gives $I(\gamma(t))\rightarrow I^*\in[0,+\infty]$,
 If $I^*<+\infty$, Zeipel's theorem tells us $\lim_{t\rightarrow t^*}\gamma(t)$ exists in finite position, 
 then we pick another calibrating curve $\eta$ defined in some $[t^*, t^*+\sigma)$ $(\sigma>0)$ with $\eta(t^*) = \lim_{t\rightarrow t^*}\gamma(t)$, 
 concatenating $\gamma, \eta$ produces a new calibrating curve defined in $[0,t^*+\sigma)$, a contradiction.
 If $I^* = +\infty$, we choose a sequence $x_n = \gamma(t_n), (t_n\rightarrow t^{*-})$ such that $\|x_n - x\|\rightarrow+\infty$ 
 and denote
 \[
 B_n:=A_h\left(\gamma\big\vert_{[0,t_n]}\right).
 \]
 Since \(t_n\le t^*<+\infty\), the Cauchy--Schwarz inequality gives
 \begin{equation}
     \|x_n-x\|^2
     =\left\|\int_0^{t_n}\dot\gamma(s)\,ds\right\|^2
     \le t_n\int_0^{t_n}\|\dot\gamma(s)\|^2\,ds
     \le 2t_n B_n.
 \end{equation}
 Because \(\gamma\) is calibrated, it is an \(h\)-minimizer; hence
 \[
 B_n=A_h\left(\gamma\big\vert_{[0,t_n]}\right)=\phi_h(x,x_n).
 \]
 By Proposition \ref{uniform bound for phi_hn}, with the energy fixed at \(h\), there is a modulus \(\mu\) such that
 \[
 \phi_h(x,x_n)\le \mu(\|x-x_n\|),
 \]
 where \(\mu(r)=O(r)\) if \(h>0\) and \(\mu(r)=O(r^{1/2})\) if \(h=0\). Therefore
 \[
 \frac{\|x_n-x\|^2}{2t_n}
 \le B_n
 =\phi_h(x,x_n)
 \le \mu(\|x-x_n\|).
 \]
 Since \(t_n\le t^*<+\infty\) while \(\|x_n-x\|\to+\infty\), the left-hand side grows quadratically in \(\|x_n-x\|\), whereas the right-hand side grows at most linearly. This is impossible.
 
 \end{proof}

\section{Acknowledgement}
This work is partially supported by National Natural Science Foundation of China (Grant No.12071316).

\end{document}